\documentclass[12pt,reqno]{amsart}
\usepackage[numbers]{natbib}
\usepackage{amsfonts}
\usepackage{amsmath}
\usepackage{amsthm}
\usepackage{mathrsfs}
\usepackage{amssymb}
\usepackage{stmaryrd}
\usepackage{graphicx}
\usepackage{epstopdf}
\usepackage{textcomp}

\newtheorem{theorem}{Theorem}
\newtheorem{lemma}{Lemma}

\newcommand{\dom}{{\rm dom\,}}
\newcommand{\co}{{\rm co\,}}
\theoremstyle{definition}

\begin{document}
\title[Optimal production and pricing strategies]{Optimal production and pricing strategies in a dynamic model of monopolistic firm}

\author{Dmitry B. Rokhlin}
\author{Georgii Mironenko}

\address{Institute of Mathematics, Mechanics and Computer Sciences,
              Southern Federal University,
Mil'chakova str., 8a, 344090, Rostov-on-Don, Russia}
\email[Dmitry B. Rokhlin]{rokhlin@math.rsu.ru}

\address{Scientific Research Institue ``Specvuzavtomatika'',
Grecheskogo Goroda Volos str., 6, 344011, Rostov-on-Don, Russia}
\email[Georgii Mironenko]{georim89@mail.ru}
\thanks{The research of the D.B.\,Rokhlin is supported by Southern Federal University, project 213.01-07-2014/07.}

\begin{abstract}
We consider a deterministic continuous time model of monopolistic firm, which chooses production and pricing strategies of a single good. Firm's goal is to maximize the discounted profit over infinite time horizon. The no-backlogging assumption induces the state constraint on the inventory level. The revenue and production cost functions are assumed to be continuous but, in general, we do not impose the concavity/convexity property. Using the results form the theory of viscosity solutions and Young-Fenchel duality, we derive a representation for the value function, study its regularity properties, and give a complete description of optimal strategies for this non-convex optimal control problem. In agreement with the results of Chazal et al. (2003), it is optimal to liquidate initial inventory in finite time and then use an optimal static strategy. We give a condition, allowing to distinguish if this static strategy can be represented by an ordinary or relaxed control. The latter is related to production cycles. General theory is illustrate by the example of a non-convex production cost, proposed by Arvan and Moses (1981).
\end{abstract}
\subjclass[2010]{91B38, 49K15, 49L25}
\keywords{Production, pricing, inventory, state constraints, constrained viscosity solution, non-convex production cost}

\maketitle

\section{introduction}
\label{sec:1}
\setcounter{equation}{0}
In the last decades in economical literature there was a considerable interest to the models of firms, performing coordinated decisions on production and pricing: see the reviews \cite{EliSte93,YanGil04,CheLev12}. In such models the ability to influence the demand by dynamic pricing may considerably change optimal inventory levels. The goal of the present paper is to study optimal strategies of a profit maximizing firm in continuous time deterministic setting.

Assume that a firm can produce some good at rate $\alpha\ge 0$. Let $C(\alpha)$ be the related price of production. Being a monopolist, the firm can settle a price $p\ge 0$ of the unit good. The demand rate $q=D(p)$ is known strictly decreasing function of price. An elementary, but natural problem is to maximize the instantaneous profit flow:
\begin{equation} \label{1.1}
 R(q)-C(q)\to\max_{q\ge 0},\quad R(q)=q D^{-1}(q).
\end{equation}
This simple model of a monopoly is well-known: see, e.g., \cite[Chapter 14]{Var92}. For a positive optimal solution $\widehat q$ of (\ref{1.1}) the marginal revenue $R'(\widehat q)$ coincides with the marginal production cost $C'(\widehat q)$.

In the present paper we are interested in continuous time extension of this model. We assume that the firm can continuously produce and sell a single good. Firm's goal is to maximize the discounted profit over the infinite time horizon. The backlogging is not allowed: that is, the inventory level should be non-negative. A preliminary formulation of the correspondent optimal control problem is the following:
\begin{equation} \label{1.2}
\int_0^\infty e^{-\beta t} (R(q_t)-C(\alpha_t))\,dt\to\max,
\end{equation}
\begin{equation} \label{1.3}
 \dot X_t=\alpha_t-q_t,\quad X_0=x\ge 0;\quad X_t\ge 0,\quad t\ge 0.
\end{equation}
Let us call this problem \emph{convex} if $R$ is concave and $C$ is convex.

In contrast to existing literature, we do not introduce any storage cost. The reason lies in the nature of optimal strategies:  even without such costs an optimal inventory path is monotonically decreasing, reaches zero in finite time and stands forever at this level. In fact, the firm needs no warehouse. It is appropriate to make a reservation that for a non-convex production cost $C$ it may be necessary to use relaxed (or randomized) production strategies to retain the inventory $X$ at zero level and meet an optimal demand. In practice, instead of the relaxed control, one may use approximately optimal strategies, corresponding to production cycles, where the inventory oscillates near $0$. So, there is a need in a ``small'' warehouse.

First continuous time production/pricing model of the same sort as (\ref{1.2}), (\ref{1.3}) was proposed in \cite{Pek74}. The horizon in \cite{Pek74} is finite and the demand curve is linear (and depends on time). In 1980-s this line of research was continued in \cite{ThoSetTen84,FeiHar85,EliSte87}. The focus of \cite{ThoSetTen84,EliSte87} was also the finite horizon case, and even a discounting was not introduced. In \cite{FeiHar85} the case of infinite horizon was addressed, but instead of the state constraint $X_t\ge 0$ the authors consider a penalty function. Prior to \cite{FeiHar85}, an infinite horizon model was considered in \cite{ArvMos81,ArvMos82}. Although the papers \cite{ArvMos81,ArvMos82} are not mathematically rigorous, they contain interesting economical insights. In particular, the authors emphasized that if the production cost is non-convex, then the ability to store the product may result in the dominance of a cyclic production strategy over a static one. The example of \cite{ArvMos82} is analyzed below.

The paper, which is most closely related to ours, is \cite{ChaJouTah03}. Although the model of \cite{ChaJouTah03} contains a strictly positive storage cost and is related to the finite horizon case, the several key conclusions remain valid for our model.
Under the assumption that the problem is convex, it was proved that an optimal strategy typically consists of three stages: (i) a redundant inventory is selling, (ii) production is switched on, but sales prevail and the inventory still decreases, (iii) the stock is cleared, the intensities of production and selling are equal and maximize the instantaneous profit (\ref{1.1}). It was also emphasized that an accumulation of inventory is not optimal, while optimal production and price paths are non-decreasing.

In concluding remarks of \cite{ChaJouTah03} it is suggested to consider non-convex production cost functions to explain the phenomenon of inventory accumulation. In fact, this is exactly the subject of the present paper. Our conclusions are somewhat mixed: in conjunction with discounting the \emph{non-convexities can explain production cycles, but not the inventory accumulation}.

Another motivation of the present study is the empirical evidence that it is not uncommon for firms to ``operate in a region of declining marginal costs'' (see \cite{Ram91}). It is mentioned in \cite{Ram91} that the non-convexities in production costs can explain ``volatility of production relative to sales''. The well-know idea is simple. Assume, e.g., that there are three possible production rates: $\alpha\in\{0,1,2\}$ with production costs $C(0)=0$, $C(1)=1$, $C(2)=3/2$. Note that $C$, extended by $C(x)=+\infty$, $x\not\in\{1,2,3\}$, is not convex. We see that to produce two units of good in two units of time, it is better to produce with rate $\alpha=2$ in the first time interval and nothing in the second one, than to produce one unit in each time interval. If the demand intensity equals to one unit, the firm needs some storage. In the dynamic model under consideration an approximately optimal strategy of similar nature will appear below.

The paper is organized as follows. In Section \ref{sec:2} we derive a representation of the value function and study its regularity properties (Theorem \ref{th:1}). It is proved that the value function $v$ is continuously differentiable and strictly concave, even if the problem is not convex. Me mention that passing to the concave (resp., convex) hull of $R$ (resp., $C$) does not change the Hamiltonian and, hence, the value function. The same value function corresponds to the problem with relaxed controls (Theorem \ref{th:2}).

The case of zero initial inventory is considered in Section \ref{sec:3}. For a convex problem optimality of a static strategy $\alpha_t=q_t=\widehat u$, where $\widehat u$ maximizes the instantaneous profit (\ref{1.1}), was proved in \cite{ChaJouTah03}. In Theorem \ref{th:3} we give a necessary and sufficient condition for optimality of this strategy in general case. To give an economical interpretation of this condition, we divide the firm into production and sales departments. A static strategy $\widehat u$ is optimal iff there exists a ``shadow price'' $\eta$ such that the intensity $\widehat u$ is optimal for both departments, trading the good at this price. The least shadow price is identified with the least minimum point of the Hamiltonian and with the marginal indirect utility $v'(0)$ of zero inventory. Certainly, the mentioned condition is satisfied for the convex problem (Theorem \ref{th:4}). If this condition is violated, we construct an optimal relaxed strategy (Theorem \ref{th:5}) and an approximately optimal ordinary strategy, which induces cyclic inventory behavior.

In Section \ref{sec:4} we present a complete description of optimal strategies for positive initial inventory $x>0$ (Theorem \ref{th:6}). The are two major stages. At the first stage the inventory is strictly decreasing and reaches zero. The duration of this stage $\tau$ is finite and is explicitly calculated. It depends only on the relative values $v'(0)$, $v'(x)$ of the marginal indirect utility $v$ and the discounting factor $\beta$. At the second stage the inventory is retained at zero level either by relaxed or by ordinary production and selling strategies, considered in Section \ref{sec:3}. It is worth mentioning that it may be necessary to involve relaxed controls only at the second stage.

In Section \ref{sec:5} we consider a simple example of convex problem with a linear production cost $C(\alpha)=c\alpha$, $\alpha\in [0,\overline\alpha]$. The production process should start before the stock is cleared if and only if the price $c$ of production of the unit good is smaller than the least shadow price $\zeta$.

In Section \ref{sec:6} we consider the example of \cite{ArvMos81}, where the demand $D$ is linear, while the production cost $C$ is concave for small values of $\alpha$ and convex for large ones. In agreement with \cite{ArvMos81,ArvMos82} there are three main cases. (i) The firm does not exist. That is, it is optimal simply to sell an initial inventory. The production stage never starts. (ii) The purely selling stage is followed by production cycles (or relaxed production strategy). (iii) The production starts before the stock is cleared, and an ordinary static strategy is optimal when the stock is cleared. As compared to \cite{ArvMos81}, we specify the exact parameter values, corresponding to these cases.

\section{Representation of the value function} \label{sec:2}
\setcounter{equation}{0}
Suppose that a firm can produce some good at rate $\alpha_t\in A$, where $A$ is a closed subset of $\mathbb R_+=[0,\infty)$. Being a monopolist, the firm can settle a price $p_t\ge 0$ of the unit good. Under the assumption that the demand rate is some known strictly decreasing function of price: $q=D(p)$, it is more convenient to think that the firm chooses dynamically the demand rate $q_t\in Q$. The set $Q\subset\mathbb R_+$ is assumed to be compact set. The inventory level $X$ satisfies the equation
\begin{equation} \label{2.1}
 X_t=x+\int_0^t (\alpha_t-q_t)\,dt,\quad t\ge 0.
\end{equation}
Suppose that the backlogging is not allowed: $X_t\ge 0$, and the sets $Q$, $A$ satisfy the following conditions:
\begin{equation} \label{2.2}
 0\in A\cap Q,\quad A\backslash\{0\}\neq\emptyset,\quad Q\backslash\{0\}\neq\emptyset.
\end{equation}
The solution of (\ref{2.1}) will also be denoted by $X^{x,\alpha,q}$.

Let $R(q)=qp=qD^{-1}(q)$ be the instantaneous revenue, and $C(\alpha)$ the instantaneous production cost. Firm's goal is to maximize the discounted profit over the infinite time horizon:
$$ \int_0^\infty e^{-\beta t} (R(q_t)-C(\alpha_t))\,dt,\quad \beta>0.$$

We assume that $R:Q\mapsto\mathbb R_+$ is continuous, $R(0)=0$, and $C:A\mapsto\mathbb R_+$ is a non-decreasing continuous function. If $A$ is unbounded, then we additionally assume that $C$ is $1$-coercive:
\begin{equation} \label{2.3}
 C(\alpha)/\alpha\to +\infty,\quad A\ni\alpha\to+\infty.
\end{equation}

Denote by $\mathscr A(x)$ the set of Borel measurable functions $\alpha:\mathbb R_+\to A$, $q:\mathbb R_+\to Q$ such that the inventory level (\ref{2.1}) in non-negative. The value function $v$ is defined by
\begin{equation} \label{2.4}
v(x)=\sup_{(\alpha,q)\in\mathscr A(x)}\int_0^\infty e^{-\beta t} (R(q_t)-C(\alpha_t))\,dt, \quad x\ge 0.
\end{equation}

Let us introduce the Hamiltonian
\begin{equation} \label{2.5}
H(z)=\widehat R(z)+\widehat C(z),\quad \widehat R(z)=\sup_{q\in Q}\{R(q)-qz\},\quad
\widehat C(z)=\sup_{\alpha\in A}\{\alpha z-C(\alpha)\}.
\end{equation}
Note, that the functions $\widehat R$, $\widehat C$ are convex and finite on $\mathbb R$. Hence, they are also continuous.
Recall (see \cite{Son86}) that a bounded uniformly continuous function $u:\mathbb R_+\mapsto\mathbb R$ is called a \emph{constrained viscosity solution} of the Hamilton-Jacobi-Bellman (HJB) equation
\begin{equation} \label{2.6}
 \beta u(x)-H(u'(x))=0,\quad x\ge 0,
\end{equation}
if for any $x>0$ (resp., $x\ge 0$) and any test function $\varphi\in C^2(\mathbb R_+)$ such that $x$ is a minimum (resp., maximum) point of $u-\varphi$ on $(0,\infty)$ (resp., on $[0,\infty)$), the inequality
$$  \beta u(x)-H(\varphi'(x))\ge 0 \quad (resp.,\ \le 0)$$
holds true. Note, that \cite{Son86} considers a minimization problem, and the definition above is modified accordingly. Using the terminology of viscosity solutions, one may rephrase this definition by saying that $u$ is a viscosity supersolution of (\ref{2.6}) on $(0,\infty)$ and a viscosity subsolution on $[0,\infty)$.

It is easy to see that a function $u\in C^1(\mathbb R_+)$ is a constrained viscosity solution of (\ref{2.6}) if and only if
\begin{align}
\beta u(x)&=H(u'(x)),\quad x>0, \label{2.7}\\
 \beta u(0)&\le H(z),\quad z\ge u'(0). \label{2.8}
\end{align}
By $u'(0)$ we mean the right derivative. To get the last inequality consider a test function $\varphi$ with $\varphi(0)=u(0)$, $\varphi'(0)=z>u'(0)$.

The results, collected in the next lemma, were proved in \cite{Son86} (Theorems 3.3, 2.1, 2.2).
\begin{lemma} \label{lem:1}
Assume that $A$ is compact. Then the value function $v$ is bounded and uniformly continuous. Moreover, $v$ is the unique constrained viscosity solution of the HJB equation (\ref{2.6}) in the class of bounded uniformly continuous functions.
\end{lemma}
Note that the assumption (A3) of \cite{Son86}, concerning the existence of an "inward pointing direction", is satisfied, since
$\sup\{\alpha-q:\alpha\in A, q\in Q\}>0,$
as follows from (\ref{2.2}).

Denote by $\mathscr M_H=\arg\min_{z\in\mathbb R} H(z)$ the set of minimum points of $H$.
\begin{lemma} \label{lem:2}
The set $\mathscr M_H\subset\mathbb R_+$ is non-empty, closed and convex.
\end{lemma}
\begin{proof}
The set $\mathscr M_H$ is closed and convex by the continuity and convexity of $H$. For $z<0$ we have
$$\widehat R(z)>\widehat R(0),\quad \widehat C(z)=-C(0)=\widehat C(0),$$
since $R$ is non-negative, $R(0)=0$, and $C$ is non-decreasing. Thus $H$, restricted to $(-\infty,0]$, attains its strict global minimum at $z=0$. For $z>0$ we have
$$\widehat R(z)\ge R(0),\quad \widehat C (z)\ge \alpha z-C(\alpha)$$
for any $\alpha>0$, $\alpha\in A$. Thus, $H(z)\to +\infty$, $z\to +\infty$. It follows that $\emptyset\neq\mathscr M_H\subset\mathbb R_+$.
\end{proof}

Let $I$ be an interval (that is, a convex set) in $\mathbb R$. Recall that a function $\psi:I\mapsto\mathbb R$, is called absolutely continuous if for any $\varepsilon>0$ there exists $\delta>0$ such that
$$ \sum_{i=1}^k |\psi(b_i)-\psi(a_i)|<\varepsilon $$
for any disjoint intervals $(a_i,b_i)$, $i=1,\dots,k$ with $[a_i,b_i]\subset I$ and
$ \sum_{i=1}^k (b_i-a_i)<\delta. $
The set of absolutely continuous functions is denoted by $AC(I)$. A function $\psi:I\mapsto\mathbb R$ is called locally absolutely continuous if $\psi \in AC([a,b])$ for any $[a,b]\subset I$. Any locally absolutely continuous function $\psi$ is a.e. differentiable and can be recovered from its derivative by the Lebesgue interval (see, e.g., \cite[Theorem 3.30]{Leo09}):
$$ \psi(x)=\psi(\overline x)+\int_{\overline x}^x\psi'(y)\,dy,\quad \overline x, x\in I.$$
Writing ``a.e.'' we always mean "almost everywhere with respect to the Lebesgue measure".

We will use the following well known result (see \cite{Nat64} (Chapter IX, Exercise 13) or \cite[Theorem 2]{Vill84}):
\begin{lemma} \label{lem:3}
Let $\psi:[a,b]\mapsto\mathbb R$ be continuous and strictly monotonic. Then $\psi^{-1}$ is absolutely continuous if and only if $\psi'\neq 0$ a.e. on $(a,b)$.
\end{lemma}

Denote by $\zeta=\min\mathscr M_H\ge 0$ the least minimum point of $H$.
\begin{theorem} \label{th:1}
The value function $v$ is bounded:
$$ v(x)\le \frac{H(0)}{\beta}=\lim_{y\to\infty} v(y)$$
and admits the following representation:
\begin{itemize}
\item[(i)] if $\zeta=0$, then
\begin{equation} \label{2.9}
v(x)=H(0)/\beta,
\end{equation}
\item[(ii)] if $\zeta>0$ then
\begin{equation} \label{2.10}
v(x)=\frac{H(\xi(x))}{\beta}=\frac{H(\zeta)}{\beta}+\int_0^x\xi(y)\,dy,
\end{equation}
where $\xi(x)$ is defined by the equation
$$x=\Psi(\xi):=-\int_\xi^{\zeta}\frac{H'(z)}{\beta z}\,dz,\quad \xi\in (0,\zeta],\quad x\ge 0.$$
\end{itemize}
In case (ii) $v$ is strictly increasing and strictly concave. Moreover, $v'$ is absolutely continuous and satisfies the conditions $v'(0)=\zeta$, $\lim_{x\to\infty} v'(x)=0.$ Finally, $v''<0$ a.e.
\end{theorem}
\begin{proof}
First we check that (\ref{2.9}), (\ref{2.10}) are constrained viscosity solutions of (\ref{2.6}).  If $\zeta=0$, then (\ref{2.9}) satisfies (\ref{2.7}), (\ref{2.8}). Hence, (\ref{2.9}) is a constrained viscosity solution of (\ref{2.6}).

Assume that $\zeta>0$. Since $H$ is convex, its derivative $H'$ exists on a set $G=(0,\zeta)\backslash D$, where $D$ is at most countable, and $H'$ is non-decreasing on $G$: see \cite{Roc70} (Theorem 25.3). Furthermore, $H'(x)\le 0$, $x\in G$ since $H$ is decreasing on $(0,\zeta)$. If $H'(x)=0$ for some $x\in G$ then $x<\zeta$ is a minimum point of $H$. Hence, $H'(x)<0$, $x\in G$.

Consider the following continuous strictly decreasing function
$$ \Psi(\xi)=-\int_\xi^{\zeta}\frac{H'(z)}{\beta z}\,dz,\quad \xi\in (0,\zeta].$$
We have $\Psi(\zeta)=0$,
$$ \Psi(0+)=-\lim_{\xi\searrow 0} \int_\xi^{\zeta}\frac{H'(s)}{\beta s}\,ds=+\infty,
$$
since $H'$ is a.e. bounded from above by a negative constant in a right neighborhood of $0$.
It follows that the formula
$$ x=\Psi(\xi)$$
correctly defines the inverse function $\xi=\Psi^{-1}:\mathbb R_+\mapsto (0,\zeta]$, which is strictly decreasing and continuous.

Moreover, since $\Psi'(\xi)<0$ a.e. on $(0,\zeta)$, the function $\xi(x)=\Psi^{-1}(x)$ is locally absolutely continuous on $\mathbb R_+$ by Lemma \ref{lem:3}. But since $\xi$ is monotone and bounded, it follows that $\xi\in AC(\mathbb R_+)$.
By Lemma \ref{lem:3} we also conclude that $\xi'(x)<0$ a.e. on $(0,\infty)$, since $\Psi=\xi^{-1}$ is locally absolutely continuous on $(0,\zeta]$.

Furthermore, by Corollary 3.50 of \cite{Leo09} we can apply the chain rule:
\begin{equation} \label{2.11}
 1=\frac{d}{dx}\Psi(\xi(x))=\frac{H'(\xi(x))}{\beta\xi(x)}\xi'(x)\quad \textrm{a.e. on } (0,\infty).
\end{equation}
The function $H(\xi)$ is absolutely continuous as a superposition of an absolutely continuous function $\xi$ an a Lipschitz continuous function $H$ (recall that $H$ is convex). Hence,
$$ H(\xi(x))=H(\zeta)+\int_0^x \frac{d}{dy}H(\xi(y))\,dy.$$
By the chain rule and formula (\ref{2.11}), we get
$$ H(\xi(x))=H(\zeta)+\int_0^x\xi'(y) H'(\xi(y))\,dy=H(\zeta)+\beta\int_0^x \xi(y)\,dy.$$

Now it is easy to see that
$$ u(x)=\frac{H(\zeta)}{\beta}+\int_0^x\xi(y)\,dy=\frac{H(\xi(x))}{\beta}$$
is a constrained viscosity solution of (\ref{2.6}). Indeed, $u'(x)=\xi(x)$, $x>0$ and
$$\beta u(x)-H(u'(x))=\beta u(x)-H(\xi(x))=0,\quad x>0.$$
The boundary condition (\ref{2.8}) is satisfied by the definition of $\zeta$:
$$ \beta u(0)=H(\zeta)\le H(z)\quad \text{for all}\quad z.$$

Note, that $u$ is strictly increasing since $u'=\xi>0$, and is strictly concave, since $\xi$ is strictly decreasing (see \cite{HirUrrLem01}, Chapter B, Theorem 4.1.4). Furthermore, if $\zeta>0$ then
$$ u(x)\le\lim_{y\to\infty} u(y)=\lim_{y\to\infty}\frac{H(\xi(y))}{\beta}=\frac{H(0)}{\beta}.$$
Other properties of the derivatives of (\ref{2.10}), mentioned in the statement of Theorem \ref{th:1}, are evident from the construction of $\xi$.

We have proved that formulas (\ref{2.9}), (\ref{2.10}) define a constrained viscosity solution of (\ref{2.6}). If $A$ is compact, then (\ref{2.9}), (\ref{2.10}) is the value function (\ref{2.4}) by the uniqueness result, stated in Lemma \ref{lem:1}.

In general case put $A_c=A\cap [0, c]$, and denote by $H_c$, $v_c$ the correspondent Hamiltonian and value function.
For $z\le 0$ we have $H_c(z)=H(z)=-C(0)$. Let
$$ \widehat\alpha(z)\in\arg\max_{a\in A}\{z\alpha-C(\alpha)\},\quad z>0. $$
If $\widehat\alpha(z)>0$, then the inequality
$$-C(0)\le \widehat\alpha(z) \left(z -\frac{C(\widehat\alpha(z))}{\widehat\alpha(z)}\right)$$
and the coercivity condition (\ref{2.3}) imply that $\widehat\alpha(z)$, $z\in [0,\overline z]$ is bounded from above for any fixed $\overline z>0$.
Thus, for any $\overline z>0$ there exists $\overline c>0$ such that
$$ \sup_{a\in A}\{z\alpha-C(\alpha)\}=\sup_{a\in A_c}\{z\alpha-C(\alpha)\}$$
and $H_c(z)=H(z)$ for $|z|\le \overline z$, $c\ge\overline c$. Taking $\overline z>\zeta$, from the convexity of $H$ we conclude that $\zeta$ is the least minimum point of $H_c$
for $c$ large enough.

Since the expressions (\ref{2.9}), (\ref{2.10}) depend only on the values of the Hamiltonian on $[0,\zeta]$, it follows that they define a constrained viscosity solution of (\ref{2.6}) with the Hamiltonians $H$ and $H_c$ for $c\ge\overline c$. But, by Lemma \ref{lem:1}, $v_c$ is the unique constrained viscosity solution of (\ref{2.6}) with the Hamiltonians $H_c$. It follows that the functions (\ref{2.9}), (\ref{2.10}) coincide with $v_c$, $c\ge\overline c$.

Clearly, $v_c\le v$. It remains to prove the reverse inequality. For any admissible strategy $(\alpha,q)\in\mathscr A(x)$ we have
\begin{equation} \label{2.12}
\frac{d}{dt}(e^{-\beta t}v_c(X_t))=e^{-\beta t}(-\beta v_c(X_t)+(\alpha_t-q_t)v_c'(X_t))\quad a.e.,
\end{equation}
where $X$ is defined by (\ref{2.1}). From the HJB equation (\ref{2.7}) we get
\begin{equation} \label{2.13}
 \beta v_c(X_t)\ge R(q_t)-q_t v_c'(X_t)+\alpha_t v_c'(X_t)-C(\alpha_t)\quad a.e.
\end{equation}
Note, that the equality (\ref{2.7}) is satisfied for $x=0$ by the continuity property. From (\ref{2.12}), (\ref{2.13}) we obtain the inequality
$$ -\frac{d}{dt}(e^{-\beta t}v_c(X_t))\ge e^{-\beta t} (R(q_t)-C(\alpha_t)) \quad a.e. $$
It follows that
$$ v_c(x)-e^{-\beta T}v_c(X_T)\ge\int_0^T e^{-\beta t} (R(q_t)-C(\alpha_t))\, dt$$
for any $T>0$. Since $v_c$ is bounded, we conclude that
$$ v_c(x)\ge \int_0^\infty e^{-\beta t} (R(q_t)-C(\alpha_t)),\quad (\alpha,q)\in\mathscr A(x).$$
Thus, $v_c\ge v$.
\end{proof}

In the course of the proof we have showed that passing from the set $A$ to $A\cap [0,c]$ does not affect the value function $v$ for $c$ large enough.

Note also that if $\zeta>0$, then the optimal discounted gain is always lower than $H(0)/\beta$. If $\zeta=0$, the discounted gain $H(0)/\beta$ can be obtained with zero initial inventory. Moreover, any initial inventory $x>0$ is useless.

For a function $f:\mathbb R\mapsto (-\infty,+\infty]$ denote by $f^*:\mathbb R\mapsto (-\infty,+\infty]$ its Young-Fenchel transform:
$$ f^*(z)=\sup_{x\in\mathbb R}\{zx-f(x)\},$$
and by $\co f$ the convex hull:
$$ (\co f)(x)=\inf\{\delta f(x_1)+(1-\delta) f(x_2):\delta\in [0,1],\ x_i\in\dom f,\ \delta x_1+(1-\delta) x_2=x\},$$
where $\dom f=\{x:f(x)<\infty\}$. For $G\subset\mathbb R$ denote by $\co G$ the intersection of all intervals, containing $G$.
The following result can be found in \cite{HirUrrLem93} (Chapter X, Proposition 1.5.4).

\begin{lemma} \label{lem:4}
Let $G\subset\mathbb R$ be a nonempty closed set and $f:G\mapsto\mathbb R$ a continuous function. Put $f(x)=+\infty$, $x\not \in G$ and assume that $f$ is $1$-coercive: $f(x)/|x|\to+\infty$, $|x|\to\infty$. Then
$$\co f=f^{**},\quad \dom (\co f)=\co G,$$
and for any $x\in\co G$ there exist $x_1, x_2\in G$ and $\delta\in (0,1)$ such that
\begin{equation} \label{2.14}
x=\delta x_1+(1-\delta) x_2,\quad (\co f)(x)=\delta f(x_1)+(1-\delta) f(x_2).
\end{equation}
\end{lemma}

The functions $C$, $-R$ satisfy the conditions of Lemma \ref{lem:4}. Put $C(\alpha)=+\infty$, $\alpha\not\in A$ and $R(q)=-\infty$, $q\not\in Q$. To unify the notation, denote by
$$\widetilde C=C^{**}=\co C,\quad \widetilde R=-(-R)^{**}=-\co(-R)$$
the closed convex (resp., concave) hull of $C$ (resp., of $R$).

Comparing with the previous notation:
$$\widehat C(z)=C^*(z)=\sup_{x\in\mathbb R}\{xz-C(x)\}=C^{***}(z)=\sup_{x\in\mathbb R}\{xz-\widetilde C(x)\},$$
\begin{align} \label{2.15}
\widehat R(z)=\sup_{x\in\mathbb R}\{R(x)-xz\}=\sup_{x\in\mathbb R}\{x\cdot (-z)-(-R(x))\}=(-R)^*(-z)\nonumber\\
=(-R)^{***}(-z)=\sup_{x\in\mathbb R}(x\cdot (-z)-(-R)^{**}(x)\}=\sup_{x\in\mathbb R}\{\widetilde R(x)-xz\},
\end{align}
we conclude that the Hamiltonian (\ref{2.5}) can be represented as follows:
\begin{equation} \label{2.16}
 H(z)=C^*(z)+(-R)^*(-z)=\sup_{x\in\mathbb R}\{xz-\widetilde C(x)\}+\sup_{x\in\mathbb R}\{\widetilde R(x)-xz\}.
\end{equation}

Let us introduce the \emph{convexified problem}:
\begin{equation} \label{2.17}
\widetilde v(x)=\sup_{(\alpha,q)\in\widetilde{\mathscr A}(x)}\int_0^\infty e^{-\beta t}(\widetilde R(q_t)-\widetilde C(\alpha_t))\,dt,
\end{equation}
where $\widetilde{\mathscr A}(x)$ is the set of Borel measurable functions $\alpha:\mathbb R_+\mapsto\co A$, $q:\mathbb R_+\mapsto \co Q$ such that $X_t^{x,\alpha,q}\ge 0$. Note, that $\widetilde C$ still satisfies condition (\ref{2.3}): see \cite[Chapter E, Proposition 1.3.9(ii)]{HirUrrLem01}. Clearly, $v\le\widetilde v$. But, since the Hamiltonian for the convexified problem is the same as for the original one (see (\ref{2.16})), by Theorem \ref{th:1} we have $\widetilde v=v$.

Let us extend the classes of production and pricing strategies. The \emph{relaxed controls} $q_t(dy)$ and $\alpha_t(dy)$ are the mappings from $[0,\infty)$ to the sets of probability measures on $Q$ and $A$ such that the functions
$$ t\mapsto\int_Q\varphi(y)\,q_t(dy),\qquad t\mapsto\int_A\varphi(y)\,\alpha_t(dy)$$
are Borel measurable for any continuous function $\varphi$. The inventory dynamics under relaxed controls is given by
$$ X_t=x+\int_0^t y\,q_t(dy)-\int_0^t y\,\alpha_t(dy).$$
The class $\mathscr A_r(x)$ of admissible relaxed controls contains those which keep $X_t$ non-negative. The related value function is defined by
\begin{equation} \label{2.18}
v_r(x)=\sup_{(\alpha,q)\in\mathscr A_r(x)}\left(\int_0^\infty e^{-\beta t}\int_Q R(y)\,q_t(dy)dt -\int_0^\infty e^{-\beta t} \int_A C(y)\,\alpha_t(dy)dt\right).
\end{equation}
We call (\ref{2.18}) the \emph{relaxed problem}.

Note, that any admissible relaxed strategy $(\alpha_t(dx),q_t(dx))$ induces an admissible ordinary strategy
$$\left(\int_A x\,\alpha_s(dx),\int_Q x\, q_s(dx)\right)\in(\co A,\co Q)$$
for the convexified problem (\ref{2.17}). Hence, the Jensen inequality implies that
\begin{align*}
&\int_0^\infty e^{-\beta t}\left(\int_Q R(x)\, q_s(dx)-\int_A C(x)\,\alpha_s(dx)\right)\,ds\\
&\le \int_0^\infty e^{-\beta t}\left(R\left(\int_Q x\, q_s(dx)\right)-C\left(\int_A x\,\alpha_s(dx)\right)\right)\,ds\le\widetilde v(x),
\end{align*}
since $\widetilde C\le C$, $\widetilde R\ge R$. Thus, $v_r\le\widetilde v$, and an evident inequality $v\le v_r$ implies the following result.
\begin{theorem} \label{th:2}
The value functions (\ref{2.4}), (\ref{2.17}), (\ref{2.18}) related to original, convexified and relaxed problems coincide:
$v=v_r=\widetilde v.$
\end{theorem}

The equality $v=v_r$ for a state constrained problem in the case of compact state and control sets was proved in \cite{Lor87}.

\section{Optimal strategies for zero initial inventory} \label{sec:3}
\setcounter{equation}{0}
In this section we consider the case of zero initial inventory: $X_0=0$. For any constant $\widehat u\in Q\cap A$, the \emph{static strategy} $\alpha_t=q_t= \widehat u$ is admissible. If it is optimal, then
\begin{equation} \label{3.1}
\widehat u\in\mathscr M:=\arg\max_{u\in Q\cap A}\{R(u)-C(u)\}.
\end{equation}
For $\eta\in\mathbb R$ put
$$\mathscr M_R(\eta)=\arg\max_{q\in Q}\{R(q)-\eta q\},\qquad
  \mathscr M_C(\eta)=\arg\max_{\alpha\in A}\{\alpha\eta-C(\alpha)\},$$
and $\mathscr M_\eta=\mathscr M_R(\eta)\cap\mathscr M_C(\eta)$. Recall that $\zeta$ is the least minimum point of $H$.

\begin{theorem} \label{th:3}
The following conditions are equivalent.
\begin{itemize}
\item[(i)] A static strategy $\alpha_t=q_t= \widehat u\in\mathscr M$ is optimal.
\item[(ii)]  $\mathscr M_\eta\neq\emptyset$ for some $\eta\in\mathbb R$.
\item[(iii)]  $\mathscr M_\zeta\neq\emptyset$.
\end{itemize}
If $\mathscr M_\eta\neq\emptyset$, then $\eta$ is a minimum point of $H$ and $\mathscr M_\eta=\mathscr M$.
\end{theorem}
\begin{proof}
By Theorem \ref{th:1} we have
\begin{equation} \label{3.3}
\beta v(0)=H(\zeta)\le H(\eta)=\sup_{q\in Q}\{R(q)-\eta q\}+\sup_{\alpha\in A}\{\alpha\eta-C(\alpha)\},\quad \eta\in\mathbb R.
\end{equation}

(ii) $\Longrightarrow$ (i). For any $\widehat u\in\mathscr M_\eta$ from (\ref{3.3}) we get
\begin{equation} \label{3.4}
v(0)\le (R(\widehat u)-C(\widehat u))/\beta=\int_0^\infty e^{-\beta t}(R(\widehat u)-C(\widehat u))\,dt.
\end{equation}
Thus, $\alpha_t=q_t= \widehat u$ is an optimal strategy and $\widehat u\in\mathscr M$.

(i) $\Longrightarrow$ (iii). If a static strategy $\alpha_t=q_t= \widehat u$ is optimal, then
$$ R(\widehat u)-C(\widehat u)=\beta\int_0^\infty e^{-\beta t}(R(\widehat u)-C(\widehat u))\,dt=\beta v(0)=H(\zeta).$$
If $\widehat u\not\in\mathscr M_\zeta$, then we get a contradiction:
$$ H(\zeta)=\widehat R(\zeta)+\widehat C(\zeta)>R(\widehat u)-\zeta\widehat u+\zeta\widehat u-C(\widehat u)=R(\widehat u)-C(\widehat u).$$

(iii) $\Longrightarrow$ (ii) is evident.

Let $\mathscr M_\eta\neq\emptyset$. Note, that for $\widehat u\in\mathscr M_\eta$ we have $H(\eta)=R(\widehat u)-C(\widehat u)$.
It follows that the inequality in (\ref{3.3}) cannot be strict, since it would imply a strict inequality in (\ref{3.4}). Thus, if $\mathscr M_\eta\neq\emptyset$, then $H(\eta)=H(\zeta)$, and $\eta$ is a minimum point of $H$.

Let $\overline u\in\mathscr M_\eta$. Then
\begin{equation} \label{3.5}
 R(u)-\eta u\le R(\overline u)-\eta \overline u,\quad \eta u-C(u)\le \eta\overline u-C(\overline u)
\end{equation}
for any $u\in Q\cap A$. After summation we get
$$ R(u)-C(u)\le R(\overline u)-C(\overline u).$$
Hence, $\overline u\in\mathscr M$.

Now take some $u\in\mathscr M$. If $u\not\in\mathscr M_\eta$, then at least one of the inequalities (\ref{3.1}) is strict and after summation we get a contradiction with the definition of $u$:
$$  R(u)-C(u)< R(\overline u)-C(\overline u).$$
Thus, both inequalities (\ref{3.5}) are, in fact, equalities and $u\in M_\eta$ along with $\overline u$.
\end{proof}

To give an economical interpretation of the condition $\mathscr M_\eta\neq\emptyset$, assume that the firm consists of independent production and sales departments. The production department sales the good to the sales department at some \emph{shadow} price $\eta$ and gets the instantaneous revenue $\eta\alpha-C(\alpha)$. The sales department get the instantaneous revenue $R(q)-\eta q$ by selling the good at the market. For each possible price $\eta$ the departments try to choose optimal strategies $\widehat\alpha$, $\widehat q$. An equilibrium $\widehat\alpha=\widehat q$ corresponds to a shadow price $\eta$ with $\mathscr M_\eta\neq\emptyset$. By Theorem \ref{th:2} a static strategy is optimal exactly when such equilibrium exists.

Similarly to an equilibrium of a mechanical system, an equilibrium shadow price $\eta$ can be found as a minimum point of the Hamiltonian $H$. By Theorem \ref{th:1}, the least shadow price $\zeta$ coincides with the marginal indirect utility $v'(0)$ of zero inventory.


For a function $f:\mathbb R\mapsto (-\infty,+\infty]$ denote by $\partial f$ its subdifferential at a point $z$:
$$ \partial f(z)=\{\gamma\in\mathbb R: f(x)\ge f(z)+\gamma (x-z),\ x\in\mathbb R\}.$$
For any lower semicontinuous convex function $f\not\equiv +\infty$ we have (see \cite[Theorem 23.5]{Roc70} or \cite[Proposition 11.3]{RockWets09})
\begin{equation} \label{3.6}
\arg\max_{x\in\mathbb R}\{zx-f(x)\}=\partial f^*(z).
\end{equation}

\begin{theorem} \label{th:4}
Assume that the sets $Q$, $A$ are convex, the function $R$ is concave and the function $C$ is  convex.
Then for any $\eta\in\mathscr M_H$ we have
$\mathscr M_\eta\neq\emptyset$.
Hence, a stationary strategy is optimal.
\end{theorem}
\begin{proof}
From the definition of $R$, $C$ it follows that each set $\mathscr M_R(\eta)$, $\mathscr M_C(\eta)$ is nonempty for any $\eta$.
If $g(x)=f(-x)$, then $\partial g(x)=-\partial f(-x)$.
Using the Moreau-Rockafellar formula (see \cite[Theorem 23.8]{Roc70}  or \cite[Theorem 3.6.3]{AusTeb03}), from (\ref{2.16}) we get
\begin{equation} \label{3.7}
0\in\partial H(\eta)=\partial C^*(\eta)-\partial(-R)^*(-\eta)=\{x-y:x\in\partial C^*(\eta), y\in\partial(-R)^*(-\eta)\}
\end{equation}
at a minimum point $\eta$ of $H$. Furthermore, by (\ref{3.6}) we have,
$$\partial(-R)^*(-\eta)=\arg\max_{x\in\mathbb R}\{-x\eta-(-R(x))\}=\arg\max_{x\in\mathbb R} \{R(x)-\eta x\}
=\mathscr M_R(\eta),$$
$$\partial C^*(\eta)=\arg\max_{x\in\mathbb R}\{x\eta -C(x)\}=\mathscr M_C(\eta).$$
 Hence, the relation (\ref{3.7}) is equivalent to the condition $0\in\mathscr M_C(\eta)-\mathscr M_R(\eta)$, and
 $\mathscr M_\eta=\mathscr M_C(\eta)\cap\mathscr M_R(\eta)\neq\emptyset.$
\end{proof}

For a convex problem, considered in Theorem \ref{th:4}, optimality of a stationary strategy in a more direct way was proved in \cite{ChaJouTah03} (Proposition 1).

To give a simple illustration of this result consider the case of constant production cost:
$C(\alpha)=c>0$, $\alpha\in [0,\overline\alpha].$
It seems reasonable that only maximal production intensity $\alpha_t=\overline\alpha$ is optimal, since it gives an additional product for free, as comparing to any other production strategy. Let $Q=[0,\overline q]$, $\overline q>0$. A static strategy, maximizing the instantaneous profit flow, is the same as a strategy $\widehat q$, maximizing the concave revenue function $R(q)$ on $[0,\overline q\wedge\overline\alpha]$, where $\overline q\wedge\overline\alpha=\min\{\overline q,\overline\alpha\}$. Clearly, $\widehat q<\overline\alpha$, if $\overline q<\overline\alpha$. However, Theorem \ref{th:3} asserts that $\widehat\alpha=\widehat q$ are optimal.

The explanation of this ``paradox" is in the fact that an optimal selling rate $\widehat q$ is covered by any production rate $\alpha\ge\widehat q$. An additional product, produced by the maximal production rate $\overline\alpha$, is not sold by an optimal strategy $\widehat q$ and remains useless.

For a non-convex problem it is possible that $\mathscr M_\zeta=\emptyset$ and there is no optimal stationary strategy of the form (\ref{3.1}). A concrete example will be given in Section \ref{sec:6}. Our next goal is the description of optimal strategies in this case.

By Theorems \ref{th:3}, \ref{th:4} for the convexified problem (\ref{2.17}) a stationary strategy $\widetilde\alpha_t=\widetilde q_t=\widetilde u$,
\begin{equation} \label{3.8}
\widetilde u\in\arg\max\{\widetilde R(u)-\widetilde C(u):u\in\co Q\cap\co A\},
\end{equation}
is optimal for zero initial inventory, and
$$\widetilde v(0)=\frac{\widetilde R(\widetilde u)-\widetilde C(\widetilde u)}{\beta}.$$
Furthermore, by Lemma 4 there exists $\gamma\in (0,1)$, $\nu\in (0,1)$, $q^i\in Q$, $\alpha^i\in A$, $i=1,2$ such that
\begin{equation} \label{3.9}
\widetilde u=\gamma q^1+(1-\gamma) q^2=\nu\alpha^1+(1-\nu)\alpha^2,
\end{equation}
\begin{equation} \label{3.10}
\widetilde R(\widetilde u)=\gamma R(q^1)+(1-\gamma) R(q^2),\quad \widetilde C(\widetilde u)=\nu C(\alpha^1)+(1-\nu) C(\alpha^2).
\end{equation}
Consider the relaxed controls
\begin{equation} \label{3.11}
 \overline q_t(dx)=\gamma\delta_{q^1}(dx)+(1-\gamma)\delta_{q^2}(dx),\quad
 \overline \alpha_t(dx)=\nu\delta_{\alpha^1}(dx)+(1-\nu)\delta_{\alpha^2}(dx),
\end{equation}
where $\delta_a$ is the Dirac measure, concentrated at $a$. They are admissible for the relaxed problem (\ref{2.18}), since
$$ \int_0^t \int_Q x\overline q_s(dx)\,dt-\int_0^t \int_A x\overline\alpha_s(dx)\,dt=t(\gamma q^1+(1-\gamma) q^2-\nu\alpha^1-(1-\nu)\alpha^2)=0.$$
By virtue of Theorem \ref{th:2}, the following simple calculation shows that the strategy (\ref{3.11}) is optimal:
\begin{align*}
v_r(0)\ge &\int_0^\infty e^{-\beta t}\left(\int_Q R(x)\,\overline q_s(dx)-\int_A C(x)\,\overline \alpha_s(dx)\right)\,ds\nonumber\\
=&\int_0^\infty e^{-\beta t}\left(\widetilde R(\widetilde u)-\widetilde C(\widetilde u)\right)\,ds=\widetilde v(0).
\end{align*}

Thus, we obtain the following result.
\begin{theorem} \label{th:5}
(i) Let $\widetilde C$, $\widetilde R$ be convex and concave hulls of $C$, $R$. Then $\widetilde u$, defined by (\ref{3.8}), is an optimal static strategy for the convexified problem (\ref{2.17}) with zero initial inventory.

(ii) There exist $q^i\in Q$, $\alpha^i\in A$, $i=1,2$ and $\gamma\in (0,1)$, $\nu\in (0,1)$ such that (\ref{3.9}), (\ref{3.10}) holds true. The relaxed static strategy (\ref{3.11}) gives the solution of the problem (\ref{2.18}) with zero initial inventory.
\end{theorem}

The strategy (\ref{3.11}) distributes the price (resp., production rate) between two levels, which are determined by the value of $\widehat u$, calculated for the convexified problem, and the location of $R$ (resp., $C$) relative to its concave (resp., convex) hull.

The randomization of production and pricing strategies, assuming by relaxed controls, can hardly be realized in practice. So, it makes sense to construct an ordinary approximately optimal strategy $(\alpha^\varepsilon,q^\varepsilon)\in\mathscr A(0)$. We will use the following elaboration of Lemma \ref{lem:4}: see \cite{HirUrrLem93} (Chapter X, Theorems 1.5.5, 1.5.6).
\begin{lemma} \label{lem:5}
Suppose that the assumptions of Lemma \ref{lem:4} are satisfied. If for given $x\in\co G$ the points $x_1$, $x_2$ in (\ref{2.14}) are different, then $\co f$ is affine on $(x_1,x_2)$ and, in addition to (\ref{2.14}), the equalities
$$ (\co f)(x_i)=f(x_i),\quad s x-(\co f)(x)=s x_i-f(x_i),\quad s\in\partial(\co f)(x),\quad i=1,2$$
hold true.
\end{lemma}

For $\widetilde u$, defined by (\ref{3.8}), by Theorems \ref{th:3}, \ref{th:4} we have
$$ \widetilde u\in\arg\max_{q\in \co Q} \{\widetilde R(q)-q\zeta\},\quad
   \widetilde u\in\arg\max_{\alpha\in\co A}\{\alpha\zeta-\widetilde C(\alpha)\}.$$
It follows that $\zeta\in\partial\widetilde C(\widetilde u)$ and, by Lemma \ref{lem:5}, there exist $\nu\in (0,1)$, $\alpha^i\in A$, $i=1,2$, satisfying (\ref{3.9}) such that
\begin{equation} \label{3.12}
 \alpha^i\zeta-C(\alpha^i)=\widetilde u\zeta-\widetilde C(\widetilde u),\quad i=1,2.
\end{equation}
Similarly, there exist $\gamma\in (0,1)$, $q^i\in Q$, $i=1,2$, satisfying (\ref{3.9}) such that
\begin{equation} \label{3.13}
 R(q^i)-q^i\zeta=\widetilde R(\widetilde u)-\widetilde u\zeta,\quad i=1,2.
\end{equation}
From (\ref{3.12}), (\ref{3.13}) for any $\varkappa\in [0,1]$ we get
\begin{align} \label{3.14}
\widetilde R(\widetilde u)-\widetilde C(\widetilde u) &= \varkappa (R(q^1)-C(\alpha^2))+(1-\varkappa)(R(q^2)-C(\alpha^1))\nonumber\\
 &+(\varkappa(\alpha^2-q^1)+(1-\varkappa)(\alpha^1-q^2))\zeta.
\end{align}

We can assume that  $\alpha^2\ge\alpha^1$, $q^2\ge q^1$ and either $\alpha^2>\alpha^1$ or $q^2>q^1$, since otherwise $\alpha^1=\alpha^2=q^1=q^2=\widetilde u$, and the ordinary stationary strategy $\alpha_t=q_t=\widetilde u$ is optimal. Put $\tau_i=\varepsilon i$, $i\in\mathbb Z_+=\{0,1,\dots\}$ and
\begin{equation} \label{3.15}
\alpha_t^\varepsilon=\sum_{i=0}^\infty\left( \alpha^2 I_{[\tau_i,\tau_i+\varkappa\varepsilon)}(t)+\alpha^1 I_{[\tau_i+\varkappa\varepsilon,\tau_{i+1})}(t)\right),
\end{equation}
\begin{equation} \label{3.16}
q^\varepsilon_t=\sum_{i=0}^\infty \left( q^1 I_{[\tau_i,\tau_i+\varkappa\varepsilon)}(t)+q^2 I_{[\tau_i+\varkappa\varepsilon,\tau_{i+1})}(t)\right),
\end{equation}
$$ \varkappa=\frac{q^2-\alpha^1}{q^2-\alpha^1+\alpha^2-q^1}\in (0,1).$$
The function $\int_{\tau_i}^t(\alpha_s^\varepsilon - q^\varepsilon_s)\,ds$ is increasing on $[\tau_i,\tau_i+\varkappa\varepsilon)$ and decreasing on $[\tau_i+\varkappa\varepsilon,\tau_{i+1})$. It is non-negative on $[\tau_i,\tau_{i+1}]$, since
$$\int_{\tau_i}^{\tau_{i+1}}(\alpha_s^\varepsilon - q^\varepsilon_s)\,ds=\varkappa\varepsilon(\alpha^2-q^1)+(1-\varkappa)\varepsilon(\alpha^1-q^2)=0$$
by the definition of $\varkappa$. We see that $(\alpha^\varepsilon,q^\varepsilon)\in\mathscr A(0)$.

Furthermore,
\begin{align*}
 \int_{\tau_i}^{\tau_{i+1}}(R(q_t^\varepsilon)-C(\alpha_t^\varepsilon))\,dt &=\frac{R(q^1)-C(\alpha^2)}{\beta} \left(e^{-\beta\tau_i}-e^{-\beta(\tau_i+\varkappa\varepsilon)}\right)\\
 &+\frac{R(q^2)-C(\alpha^1)}{\beta} \left(e^{-\beta(\tau_i+\varkappa\varepsilon)}-e^{-\beta\tau_{i+1}}\right).
\end{align*}
Summing up these expressions, we get
\begin{align*}
 &\lim_{\varepsilon\to 0}\int_0^\infty e^{-\beta t} (R(q_t^\varepsilon)-C(\alpha_t^\varepsilon))\,dt =\lim_{\varepsilon\to 0}\frac{1}{1-e^{-\beta\varepsilon}}  \biggl( \frac{R(q^1)-C(\alpha^2)}{\beta} \left(1-e^{-\beta\varkappa\varepsilon}\right)\nonumber\\
 &+\frac{R(q^2)-C(\alpha^1)}{\beta} \left(e^{-\beta\varkappa\varepsilon}-e^{-\beta\varepsilon}\right)\biggr)\nonumber\\
 &=\frac{R(q^1)-C(\alpha^2)}{\beta}\varkappa+\frac{R(q^2)-C(\alpha^1)}{\beta}(1-\varkappa)=\frac{\widetilde R(\widetilde u)-\widetilde C(\widetilde u)}{\beta}=\widetilde v(0),
\end{align*}
where we used (\ref{3.14}) in the next to last equality.

We see that the strategy (\ref{3.15}), (\ref{3.16}) is approximately optimal:
$$ v(0)=\lim_{\varepsilon\to 0}\int_0^\infty e^{-\beta t}(R(q_t^\varepsilon)-C(\alpha_t^\varepsilon))\,dt. $$
Note also, that under this strategy the inventory level $X_t$ demonstrates cyclic behavior and $0\le X_t\le\varkappa(\alpha^2-q^1)\varepsilon$.

\section{Optimal strategies for positive initial inventory} \label{sec:4}
\setcounter{equation}{0}
A complete description of optimal strategies are given in Theorem \ref{th:6} below. In its proof we use the following result.
\begin{lemma} \label{lem:6}
Let $f$ satisfy the assumptions of Lemma \ref{lem:4}, and let $F$ be a co-countable set, where the function $f^*$ is differentiable. In view of (\ref{3.6}) put
$$ \{\widehat x(z)\}=\{(f^*)'(z)\}=\arg\max_{x\in\mathbb R}\{xz-\co f(x)\},\quad z\in F.$$
Then
\begin{equation} \label{4.1}
\widehat x(z)\in\dom f,\quad \co f(\widehat x(z))=f(\widehat x(z))
\end{equation}
for $z\in F$.
\end{lemma}
\begin{proof}
If (\ref{4.1}) is not true, then from Lemma \ref{lem:5} it follows that $\co f$ is affine in a neighbourhood of $\widehat x(z)$. Hence, the set $\arg\max_{x\in\mathbb R}\{xz-\co f(x)\}$ contains this neighbourhood, and $z\not\in F$.
\end{proof}

If $\zeta=0$, then, by Theorem \ref{th:1}, $v$ is constant. This case is somewhat trivial, since an optimal strategy for zero initial inventory $X_0=0$ retains this property for $X_0>0$. Thus, we will assume that $\zeta>0$.

The convex functions $\widehat R$, $\widehat C$ are differentiable on a co-countable subset $F$ of $(0,\zeta)$. Hence, by (\ref{3.6}) we conclude that each of the sets
$$ \widetilde{\mathscr M}_R(z)=\arg\max_{q\in \co Q} \{\widetilde R(q)-qz\},\quad
   \widetilde{\mathscr M}_C(z)=\arg\max_{\alpha\in\co A}\{\alpha z-\widetilde C(\alpha)\}$$
contains exactly one point:
\begin{align}
\widetilde{\mathscr M}_R(z)&=\arg\max_{q\in \co Q}(\widetilde R(q)-zq\}=\arg\max_{x\in\mathbb R}(-zx-(-\widetilde R(x))=\partial (-\widetilde R)^*(-z)\nonumber\\
&=-\partial\widehat R(z)=\{-\widehat R'(z)\},\label{4.2}\\
\widetilde{\mathscr M}_C(z)&=\arg\max_{\alpha\in \co A}(z\alpha-\widetilde C(\alpha)\}=\partial C^*(z)=\{\widehat C'(z)\}\nonumber
\end{align}
for $z\in F$. In (\ref{4.2}) we used the equality $\widehat R(z)=(-R)^*(-z)=(-\widetilde R)^*(-z)$: see (\ref{2.15}).

\begin{theorem} \label{th:6}
Let $F\subset (0,\zeta)$ be a co-countable set, where the convex functions $\widehat R$, $\widehat C$ are differentiable. Put
$$ \{\widehat q(z)\}=\arg\max_{q\in \co Q} \{\widetilde R(q)-q z\},\quad
   \{\widehat\alpha(z)\}=\arg\max_{\alpha\in\co A}\{\alpha z-\widetilde C(\alpha)\},\quad z\in F,$$
$$\widehat u\in\arg\max(\widetilde R(u)-\widetilde C(u):u\in\co Q\cap\co A).$$
For given initial inventory $x>0$ put
$$\tau=\frac{1}{\beta}\ln\frac{v'(0)}{v'(x)}$$
and define $X$ by the equation
\begin{equation} \label{4.3}
v'(X_t)=v'(x)e^{\beta t},\quad t\in [0,\tau].
\end{equation}
Furthermore, put $\mathscr T=\{t\in [0,\tau]:v'(X_t)\in F\}$ and consider the strategy
 \begin{equation} \label{4.4}
 \alpha^*_t=\widehat\alpha(v'(X_t)),\quad q^*_t=\widehat q(v'(X_t)),\quad t\in\mathscr T,
\end{equation}
\begin{equation} \label{4.5}
 \alpha^*_t=q^*_t=\widehat u,\quad t>\tau.
\end{equation}
On the countable set $[0,\tau]\backslash \mathscr T$ the values $\alpha^*_t$, $q_t^*$ can be defined in an arbitrary way.

(i) The strategy (\ref{4.4}) is optimal for the convexified problem (\ref{2.17}).

(ii) We have
\begin{equation} \label{4.6}
(\alpha_t^*,q_t^*)\in\dom C\times\dom R,\quad \widetilde C(\alpha^*_t)=C(\alpha^*_t),\quad \widetilde R(\alpha_t^*)=R(\alpha_t^*),\quad t\in\mathscr T.
\end{equation}

(iii) There exist $q^i\in Q$, $\alpha^i\in A$, $i=1,2$ and $\gamma\in (0,1)$, $\nu\in (0,1)$ such that (\ref{3.9}), (\ref{3.10}) hold true. Replacing (\ref{4.5}) by the static relaxed control
\begin{equation} \label{4.7}
\overline q(dx)=\gamma\delta_{q^1}(dx)+(1-\gamma)\delta_{q^2}(dx),\quad
 \overline \alpha(dx)=\nu\delta_{\alpha^1}(dx)+(1-\nu)\delta_{\alpha^2}(dx),
\end{equation}
we get the solution of the relaxed problem (\ref{2.18}).

(iv) If $\mathscr M_\zeta\neq\emptyset$, then replacing (\ref{4.5}) by
\begin{equation} \label{4.8}
\widehat u\in\arg\max_{u\in Q\cap A}(R(u)-C(u)),
\end{equation}
we get an  optimal solution of the problem (\ref{2.4}).
\end{theorem}
\begin{proof}
Since $H$ is differentiable on $F$ and $v':(0,\infty)\mapsto (0,\zeta)$ is a bijection, it follows that $H'(v'(x))$ is well-defined on the co-countable set $(v')^{-1}(F)\subset (0,\infty)$. The second derivative $v''$ exists a.e. on $(0,\infty)$. By the chain rule (see \cite[Corollary 3.48]{Leo09}), from the HJB equation it follows that
\begin{equation} \label{4.9}
 \beta v'(x)=H'(v'(x))v''(x)\quad\textrm{a.e. on } (0,\infty).
\end{equation}

Put $\widehat q(z)=-\widehat R'(z)$, $\widehat\alpha(z)=\widehat C'(z)$, $z\in F$. Then
\begin{equation} \label{4.10}
 H'(z)=\widehat R'(z)+\widehat C'(z)=-\widehat q(z)+\widehat\alpha(z)<0,\quad z\in F.
\end{equation}
Using $\widehat q(v'(x))$, $\widehat\alpha(v'(x))$ formally as \emph{feedback controls}, by formulas (\ref{4.9}), (\ref{4.10}) we get
$$ dt=\frac{d X_t}{\widehat\alpha(v'(X_t)-\widehat q(v'(X_t))}=\frac{dX_t}{H'(v'(X_t))}=\frac{1}{\beta}d (\ln v'(X_t)). $$
Using the initial condition $X_0=x$, after the integration we get
\begin{equation} \label{4.11}
 t=\frac{1}{\beta}\ln\frac{v'(X_t)}{v'(x)}.
\end{equation}
The function
$$ y\mapsto\frac{1}{\beta}\ln\frac{v'(y)}{v'(x)}:[0,x]\mapsto [0,\tau]$$
is a bijection. Define the function $t\mapsto X_t:[0,\tau]\mapsto[0,x]$ by the equation (\ref{4.3}), which is identical to (\ref{4.11}). Since $v''<0$ a.e., from Lemma \ref{lem:3} it follows that the strictly decreasing function $t\mapsto X_t$ is absolutely continuous.

Let us proof that the strategy (\ref{4.4}), (\ref{4.5}) is optimal for the convexified problem (\ref{2.17}).
Since $v'$ is absolutely continuous (see Theorem \ref{th:1}), the function $\ln v'(y)$ is also absolutely continuous on $[0,x]$. Hence,
$$ \ln v'(x)-\ln v'(y)=\int_y^x\frac{d}{dz}\ln v'(z)\,dz=\int_y^x\frac{v''(z)}{v'(z)}\,dz=\beta\int_y^x\frac{dz}{H'(v'(z))}, \quad y\in [0,x]$$
and
\begin{equation} \label{4.12}
\frac{1}{\beta}\ln\frac{v'(x)}{v'(X_t)}=-t=-\int_{X_t}^x\frac{dz}{H'(v'(z))},\quad t\in [0,\tau].
\end{equation}
From (\ref{4.12}) by the chain rule \cite[Corollary 3.48]{Leo09} we get
$$\dot X_t=H'(v'(X_t))=\widehat\alpha(v'(X_t)-\widehat q(v'(X_t))=\alpha^*_t-q^*_t\quad\textrm{a.e. on } [0,\tau].$$
Here we used (\ref{4.10}) and (\ref{4.4}).
For $t>\tau$ the relation $\dot X_t=\alpha^*_t-q^*_t=0$ is trivially satisfied. We conclude that the strategy $(\alpha^*,q^*)$ is admissible, since it induces a non-negative state process $X=X^{x,\alpha^*,q^*}$.

To prove that $(\alpha^*,q^*)$ is optimal, it is enough to show that the function
\begin{equation} \label{4.13}
W(t)=\int_0^t e^{-\beta s}(\widetilde R(q_s^*)-\widetilde C(\alpha_s^*))\,ds+e^{-\beta t} v(X_t^{x,\alpha^*,q^*})
\end{equation}
is constant, since then
$$ W(0)=v(x)=\lim_{t\to\infty}W(t)=\int_0^\infty e^{-\beta s}(\widetilde R(q_s^*)-\widetilde C(\alpha_s^*))\,ds.$$
Differentiating (\ref{4.13}), by (\ref{4.4}) we get
\begin{align*}
\dot W &=e^{-\beta t}\left(\widetilde R(q_t^*)-\widetilde C(\alpha_t^*)-\beta v(X_t^{x,\alpha^*,q^*})+v'(X_t^{x,\alpha^*, q^*})(\alpha_t^*- q_t^*)\right)\\
 &=e^{-\beta t}\left(-\beta v(X_t^{x,\alpha^*,q^*})+\widehat R(v'(X_t^{x,\alpha^*, q^*}))+\widehat C(v'(X_t^{x,\alpha^*, q^*}))\right)\\
 &=e^{-\beta t}\left(-\beta v(X_t^{x,\alpha^*,q^*})+H(v'(X_t^{x,\alpha^*,q^*}))\right)=0\quad \textrm{a.e. on } (0,\tau).
\end{align*}
For $t>\tau$ we have $X_t^{x,\alpha^*,q^*}=0$, $\alpha^*_t=q^*_t=\widehat u$ and
$$ W(t)=\int_0^\tau e^{-\beta s}(\widetilde R(q_s^*)-\widetilde C(\alpha_s^*))\,ds+\frac{1}{\beta}e^{-\beta\tau}(\widetilde R(\widehat u)-\widetilde C(\widehat u)),$$
since $-e^{-\beta t}(\widetilde R(\widehat u)-\widetilde C(\widehat u))/\beta+e^{-\beta t}v(0)=0$ by the optimality of $\widehat u$ for zero initial inventory: see Theorem \ref{th:4}.

Since $v'(X_t)\in F$, $t\in\mathscr T$, by Lemma \ref{lem:6} we infer that the relations (\ref{4.6}) hold true. It follows that
$$ v(x)=\int_0^\tau e^{-\beta s}(R(q_s^*)-C(\alpha_s^*))\,ds+e^{-\beta \tau} v(0)$$
since $W$ is constant. From Theorems \ref{th:5} and \ref{th:4} it follows that the strategies (\ref{4.4}), (\ref{4.7}) and (\ref{4.4}), (\ref{4.8}) (under the condition $\mathscr M_\zeta\neq \emptyset$) give the same value of the objective functional as (\ref{4.4}), (\ref{4.5}). Hence, they are optimal for the problems (\ref{2.18}) and (\ref{2.4}) respectively.
\end{proof}

Note, that by (\ref{4.10}) the state process $X^{x,\alpha^*,q^*}$, induced by (\ref{4.4}) is strictly decreasing on $(0,\tau)$.
In the case of $\mathscr M_\zeta=\emptyset$ one can also use an approximately optimal strategy (\ref{3.15}), (\ref{3.16}) on $(\tau,\infty)$ instead of an optimal relaxed strategy (\ref{4.7}).

\section{The case of strictly concave revenue and linear production cost}
\label{sec:5}
\setcounter{equation}{0}
Let $Q=[0,\overline q]$, $A=[0,\overline\alpha]$, $\overline q$, $\overline\alpha>0$. Assume that $R$ is differentiable and strictly concave, and let $C=c\alpha$, $c>0$. An optimal strategy is completely described by the functions $\widehat q(v'(x))$, $\widehat\alpha(v'(x))$ and the value $\widehat u$, given in Theorem \ref{th:6}.

We have
$$ H(z)=\sup_{q\in [0,\overline q]}(R(q)-qz)+\sup_{\alpha\in [0,\overline\alpha]}(\alpha z-C(\alpha))
=\widehat R(z)+\overline\alpha(z-c)^+,\quad x^+=\max\{0,x\}.$$
Since $\mathscr M_R(z)$ contains exactly one element $\widehat q(z)$, the function $\widehat R$ is continuously differentiable (see \cite[Theorem 25.5]{Roc70}), and $\widehat q(z)=-\widehat R'(z)$ is continuous.

If $R'(0)\le 0$, then $\widehat q(z)=0$, $z\ge 0$ and $\zeta=\min\mathscr M_H=\{0\}$. From theorem \ref{th:1} it follows that $v(x)=H(0)/\beta=0$ and the problem is trivial. Thus, we may assume that $R'(0)>0$.

The calculation of an optimal static strategy  (\ref{3.1}) for zero initial inventory is simple:
$$ \{\widehat u\}=\arg\max_{u\in [0,\overline\alpha\wedge\overline q]}(R(u)-cu)=\begin{cases}
0,& R'(0)\le c,\\
(R')^{-1}(c),& R'(0)>c, R'(\overline\alpha\wedge\overline q)<c,\\
\overline\alpha\wedge\overline q,& R'(\overline\alpha\wedge\overline q)\ge c.
\end{cases}$$
Furthermore, from the formula
\begin{equation} \label{5.1}
 H'(z)=\begin{cases}
-\widehat q(z),& z\in (0,c),\\
\overline\alpha-\widehat q(z),& z>c
\end{cases}
\end{equation}
it follows that $\widehat q$ is non-increasing, since $H$ is convex. Clearly, $\lim_{z\to+\infty}\widehat q(z)=0$ (and $\widehat q(z)=0$ for sufficiently large $z$, if $R'(0)<+\infty$). From (\ref{5.1}) we conclude that
$$ \zeta=\min\mathscr M_H=\inf\{z\ge 0:\widehat q(z)=0\}\wedge\inf\{z\ge c:\widehat q(z)\le\overline\alpha\}>0.$$
The function $\widehat q(v'(x))$, $x\in (0,\infty)$ is non-decreasing and
$$ \lim_{x\to 0}\widehat q(v'(x))=\widehat q(v'(0))=\widehat q(\zeta)=\widehat u,$$
$$ \lim_{x\to +\infty}\widehat q(v'(x))=\widehat q(0)=\arg\max_{q\in [0,\overline q]} R(q),$$
by Theorems \ref{th:1} and \ref{th:3}.

The function $\widehat \alpha(v'(x))$, $x\in (0,\infty)$ is piecewise constant:
$$ \widehat\alpha(v'(x))=\begin{cases}
\overline\alpha,& v'(x)>c,\\
0,& v'(x)<c.
\end{cases}$$
In particular, $\widehat\alpha(v'(x))=0$, $x>0$ if $\zeta=v'(0)<c$. Note, that in the opposite case, where the production price $c$ is lower than the marginal indirect utility of zero inventory: $\zeta=v'(0)>c$, the production should start when the inventory is reduced to the level $\widehat x>0$, defined by the equation $v'(\widehat x)=c$. However, after the production is switched on, the inventory level still decreases and reaches zero in finite time $\tau$.

From (\ref{4.9}) it follows that $v$ is twice continuously differentiable on $(0,\widehat x)\cup(\widehat x,+\infty)$, and $v''=\beta v'/H'(v')$ on this set. However, $v''$ is discontinuous at $\widehat x$:
$$ v''(\widehat x+0)=\frac{\beta c}{H'(c+)}\neq v''(\widehat x-0)=\frac{\beta c}{H'(c-)}.
$$
Thus, $v''$ can be discontinuous under assumptions of Theorem \ref{th:1}.

\section{Arvan-Moses example}
\label{sec:6}
\setcounter{equation}{0}
In the example of \cite{ArvMos81} the demand is linear, hence the revenue looks as follows:
$$ R(q)=(A-Bq)q,\quad q\in [0,A/B].$$
The production cost
$$ C(\alpha)=\alpha^3/3-K\alpha^2+K^2\alpha,\quad \alpha\ge 0$$
is concave on $[0,K]$ and convex on $[K,\infty)$. It is assumed that where $A, B, K>0$.
We have
$$ R'(q)=A-2Bq,\quad R''(q)=-2B,$$
$$ C'(\alpha)=(\alpha-K)^2,\quad C''(\alpha)=2(\alpha-K).$$
Note, that the function $C$ is strictly increasing and $1$-coercive. The function $R$ is strictly concave.

Take the largest $s$ such that $s\alpha\le C(\alpha)$, $\alpha\ge 0$. We have
$$ s=\inf_{\alpha\ge 0}\{\alpha^2/3-K\alpha+K^2\}=\frac{K^2}{4},$$
where infimum is attained at $\alpha_0=3K/2$. It is easy to see that
$$ \widetilde C(\alpha)=(\co C)(\alpha)=\begin{cases}
K^2\alpha/4,&\alpha\in [0,3K/2],\\
C(\alpha),&\alpha\ge 3K/2.
\end{cases}$$

For the convexified problem (\ref{2.17}) an optimal static strategy for zero initial inventory is given by
$$ \widehat u\in\arg\min_{u\in [0,A/B]}(R(u)-\widetilde C(u)).$$
We have
$$ R'(u)-\widetilde C'(u)=\begin{cases}
A-2Bu-K^2/4,& u<3K/2,\\
A-2Bu-(u-K)^2,& u>3K/2.
\end{cases}$$
It follows that
\begin{equation} \label{6.1}
\widehat u=\begin{cases}
0,& A\le K^2/4,\\
(A-K^2/4)/(2B),& K^2/4\le A\le 3BK+K^2/4,\\
-B+K+\sqrt{B^2-2BK+A},& A\ge 3BK+K^2/4.
\end{cases}
\end{equation}

For $A\le K^2/4$ and $A\ge 3BK+K^2/4$ we have $\widetilde C(\widehat u)=C(\widehat u)$. Hence, in these cases $\alpha_t=q_t=\widehat u$ is an optimal solution for the original problem (\ref{3.4}). If
\begin{equation} \label{6.2}
\frac{K^2}{4}< A< 3BK+\frac{K^2}{4},
\end{equation}
then $\widehat u$ belongs to the interval $(0,3K/2)$, where $\widetilde C$ is linear and $\widetilde C<C$. In this case an optimal relaxed production control is constructed by formulas (\ref{3.9}) -- (\ref{3.11}):
\begin{equation} \label{6.3}
\overline\alpha_t(dx)=\nu\delta_0+(1-\nu)\delta_{3K/2},\quad (1-\nu)\frac{3K}{2}=\widehat u=\left(A-\frac{K^2}{4}\right)\frac{1}{2B}.
\end{equation}

We see that for zero initial inventory an ordinary static strategy is not optimal iff the condition (\ref{6.2}) is satisfied. In this case instead of the relaxed strategy (\ref{6.3}) one can use an approximately optimal strategy (\ref{3.15}), (\ref{3.16}). Under this strategy the inventory remains close to $0$: $X_t\le b\varepsilon$, $b>0$ and demonstrates cyclic accumulation-decumulation behavior, described in \cite{ArvMos81}. However, it need not produce discounted profit close to optimal if accumulation-decumulation cycles are not small. It is also interesting to note that the condition
$$ K^2< A< 3BK+\frac{K^2}{4},$$
quite similar to (\ref{6.2}), appeared in \cite{ArvMos81}.

Let $z\ge 0$. Denote by $\widehat\alpha$, $\widehat q$ the maximum points of
$$ z\alpha-\widetilde C(\alpha)\to\max_{\alpha\ge 0},\quad R(q)-z q\to\max_{q\in [0,A/B]}.$$
We have
\begin{equation} \label{6.4}
\widehat\alpha(z)=\begin{cases}
    0,      &\quad z\in [0,K^2/4),\\
    K+\sqrt z, & \quad z> K^2/4,
  \end{cases}\quad
\widehat q(z)=\begin{cases}
    (A-z)/(2B),      &\quad z\in [0,A],\\
    0,       & \quad z\ge A,
  \end{cases}
\end{equation}
and $\widehat\alpha(K^2/4)\in [0,3K/2]$. Let us find the least mimimum point $\zeta=\min\mathscr M_H$ of the Hamiltonian
\begin{align*}
H(z)&=\sup_{q\in [0, A/B]}\{(A-B q) q-zq\}+\sup_{\alpha\ge 0}\{z\alpha-\widetilde C(\alpha)\}\\
 &=\frac{(z-A)^2}{4B} I_{\{z< A\}}+(z\widehat\alpha(z)-C(\widehat\alpha(z))) I_{\{z> K^2/4\}},\quad z\ge 0.
\end{align*}
If $A\le K^2/4$, then $\zeta=A$. Otherwise, consider
$$ H'(z)=\frac{z-A}{2B}+\widehat\alpha(z)=\frac{z-A}{2B}+K+\sqrt z,\quad z\in \left(\frac{K^2}{4},A\right).$$
Since $\lim_{z\nearrow A} H'(z)>0$, it follows that $\zeta\in (K^2/4,A)$ iff
$$ \lim_{z\searrow K^2/4}=\frac{1}{2B}\left(\frac{K^2}{4}-A+3BK\right)<0$$
Under this condition $\zeta$ is defined by the equation $H'(\zeta)=0$, $\zeta\in (K^2/4,A)$. Otherwise, $\zeta=K^2/4$. Collecting all cases, considered above, we get
$$\zeta=\begin{cases}
A,& A\le K^2/4,\\
K^2/4,& K^2/4\le A \le K^2/4+3BK,\\
\left(-B+\sqrt{B^2-2BK+A}\right)^2,& A\ge K^2/4+3BK.
\end{cases}
$$
Note, that these three cases are the same as in (\ref{6.1}).

Let us consider the feedback strategies $\widehat q(v'(x))$, $\widehat\alpha(v'(x))$, defined in Theorem \ref{th:6}. Since $v'(x)\le v'(0)=\zeta$ and $\zeta\le A$ for any set of parameters, we get
$$ \widehat q(v'(x))=\frac{A-v'(x)}{2B},\quad x>0.$$
Hence, $\widehat q(v'(x))$, $x>0$ is a strictly increasing positive function. Moreover, it is easy to see that
$$ \lim_{x\searrow 0}\widehat q(v'(x))=\frac{A-\zeta}{2B}=\widehat u.$$

If $A\le K^2/4+3BK$, then $\zeta\le K^2/4$ and
$$ \widehat\alpha(v'(x))=0,\quad x>0.$$
If $A>K^2/4+3BK$, then $v'(0)=\zeta>K^2/4$ and there exists a unique point $\widehat x>0$ such that $v'(\widehat x)=K^2/4$. In this case
$$ \widehat\alpha(v'(x))=\begin{cases}
K+\sqrt{v'(x)},& x<\widehat x,\\
0,& x>\widehat x.
\end{cases}$$

Thus, we have three cases. (i) If $A\le K^2/4$, then the firm should optimally sell the initial inventory:
$$ \dot X_t=-\widehat q(v'(X_t))=-\frac{A-v'(X_t)}{2B}<0,\quad X_t>0.$$
The production is shut down.

(ii) If $\quad K^2/4<A< K^2/4+3BK$, then the production starts after selling the initial inventory, and the relaxed  production strategy (\ref{6.3}) should meet the demand $\widehat q(\zeta)=\widehat u=(A-K^2/4)/2B$.

(iii) If $A\ge K^2/4+3BK$, then the production starts after the inventory falls below $\widehat x$. The inventory still decreases to $0$
and stabilizes at this level. Optimal demand and production intensities of the related stable regime are equal to $\widehat u$, defined by (\ref{6.1}).

Finally, to illustrate the result of Theorem \ref{th:3} we will show that $\mathscr M_\eta\neq\emptyset$ iff (\ref{6.2}) is satisfied. Let $\varphi(\alpha,z)=z\alpha-C(\alpha)$. Considering
$$ \frac{\partial\varphi}{\partial\alpha}(\alpha,z)=z-(\alpha-K)^2,\quad \alpha>0,$$
we infer that $\varphi(\cdot,z)$ has two local maximum points $\alpha_1=0$, $\alpha_2=K+\sqrt z$ for $z\in (0, K^2)$ and the global maximum point $\alpha_2=K+\sqrt z$ for $z\ge K^2$. Furthermore,
$$\frac{d}{dz}\varphi(\alpha_2(z),z)=\alpha_2(z)+(z-(\alpha_2(z)-K)^2)\alpha_2'(z) =\alpha_2(z)>0,\quad z>0$$
and $\varphi(\alpha_2(K^2/4),K^2/4)=\varphi(3K/2,K^2/4)=0.$
It follows that
$$\varphi(\alpha_1(z),\eta)=0>\varphi(\alpha_2(z),z)\quad \textrm{iff}\quad z<K^2/4$$
and
$$ \mathscr M_C(z)=\arg\max_{\alpha\ge 0}(z\alpha-C(\alpha))=
\begin{cases}
0,& z\in [0,K^2/4),\\
\{0,3K/2\},& z=K^2/4,\\
K+\sqrt z,& z>K^2/4.
\end{cases}$$
Clearly, $\mathscr M_\eta\neq\emptyset$, $\eta\ge 0$ iff
$\widehat q(\eta)\not\in\mathscr M_C(\eta)$, $\eta\ge 0,$ where $\widehat q(\eta)$ is defined by (\ref{6.4}).
It is easy to see that this is the case iff $A>K^2/4$ and
$$\widehat q(K^2/4)=\frac{A-K^2/4}{2B}<\frac{3K}{2},$$
which is the same as the right inequality in (\ref{6.2}). Thus, Theorem \ref{th:3} again implies that there is no optimal ordinary stationary strategy for zero initial inventory iff (\ref{6.2}) is satisfied.



  \bibliographystyle{plain}
  \bibliography{litInventory}





\end{document}